\documentclass{article}

\usepackage{titling}

\usepackage{amsfonts,amsmath,amssymb,amsthm}
\usepackage{mathscinet}
\usepackage{mathrsfs}

\usepackage{combelow}

\newcommand{\email}[1]{\href{mailto:#1}{#1}}

\usepackage[pdfusetitle]{hyperref}
\usepackage{color}
\definecolor{darkred}{rgb}{.8,0,0}
\definecolor{tocolor}{rgb}{.1,.1,.1}
\definecolor{urlcolor}{rgb}{.2,.2,.6}
\definecolor{linkcolor}{rgb}{.1,.1,.5}
\definecolor{citecolor}{rgb}{.4,.2,.1}
\definecolor{gray}{rgb}{.8,.8,.8}

\hypersetup{
	backref=true,
	colorlinks=true,
	urlcolor=urlcolor,
	linkcolor=linkcolor,
	citecolor=citecolor
	}

\usepackage{aliascnt}

\newcommand{\thdef}[2]{
	\newaliascnt{#1}{theorem}
	\newtheorem{#1}[#1]{#2}
	\aliascntresetthe{#1}
	\newtheorem*{#1*}{#2}
	\expandafter\newcommand\expandafter{\csname #1autorefname\endcsname}{#2}
}

\newtheorem{theorem}{Theorem}[section]
\newtheorem*{theorem*}{Theorem}
\thdef{conjecture}{Conjecture}
\thdef{lemma}{Lemma}
\thdef{corollary}{Corollary}
\thdef{proposition}{Proposition}
\theoremstyle{definition}
\thdef{definition}{Definition}

\theoremstyle{remark}
\thdef{remarkx}{Remark}
\thdef{examplex}{Example}

\newenvironment{example}
  {\pushQED{\qed}\examplex}
  {\popQED\endexamplex}

\newenvironment{remark}
  {\pushQED{\qed}\remarkx}
  {\popQED\endremarkx}

\newcommand{\defn}[1]{\textbf{\emph{#1}}}

\usepackage{xypic}

\newcommand{\rbrac}[1]{\left(#1\right)}

\newcommand{\CC}{\mathbb{C}}
\newcommand{\PP}{\mathbb{P}}

\newcommand{\ZZ}{\mathbb{Z}}

\newcommand{\T}{\mathsf{T}}
\newcommand{\W}{\mathsf{W}}
\newcommand{\Z}{\mathsf{Z}}
\newcommand{\Y}{\mathsf{Y}}
\newcommand{\X}{\mathsf{X}}
\newcommand{\tX}{\widetilde{\mathsf{X}}}
\newcommand{\tY}{\widetilde{\mathsf{Y}}}
\newcommand{\U}{\mathsf{U}}
\newcommand{\bL}{\mathsf{L}}
\newcommand{\bLt}{\widetilde{\mathsf{L}}}
\newcommand{\HolF}{\mathsf{Hol}(\cF)}
\newcommand{\B}{\mathsf{B}}
\newcommand{\G}{\mathsf{G}}
\newcommand{\N}{\mathsf{N}}

\newcommand{\tx}{\widetilde{x}}
\newcommand{\ty}{\widetilde{y}}
\newcommand{\tgamma}{\widetilde\gamma}

\newcommand{\Aut}[1]{\mathsf{Aut}\rbrac{#1}}
\newcommand{\Autom}[2]{\mathsf{Aut}_{#1}(#2)}

\DeclareMathOperator{\rank}{rank}
\DeclareMathOperator{\img}{img}
\DeclareMathOperator{\hol}{hol}
\newcommand{\Ad}{\mathrm{Ad}}

\newcommand{\pis}{\pi^\sharp}
\newcommand{\piL}{\pi_\bL}
\newcommand{\omL}{\eta}

\newcommand{\piF}{\pi_\cF}
\newcommand{\piG}{\pi_\cG}
\newcommand{\omF}{\omX_\cF}
\newcommand{\omG}{\omX_\cG}

\newcommand{\piX}{\pi}
\newcommand{\piXs}{\piX^\sharp}
\newcommand{\omX}{\sigma}
\newcommand{\omXf}{\omX^\flat}
\newcommand{\omXo}{\omX_0}
\newcommand{\omXof}{\omX_0^\flat}

\newcommand{\cT}[1]{\mathcal{T}_{#1}} % tangent sheaf
\newcommand{\Tpol}[2][\bullet]{\wedge^{#1} \cT{#2}} % polyvectors
\newcommand{\forms}[2][\bullet]{\Omega^{#1}_{#2}} % differential forms
\newcommand{\cF}{\mathcal{F}}
\newcommand{\cG}{\mathcal{G}}

\newcommand{\cO}[1]{\mathcal{O}_{#1}} % structure sheaf

\newcommand{\coH}[2][\bullet]{\mathsf{H}^{#1}(#2)}
\newcommand{\End}[1]{\mathsf{End}(#1)}
\newcommand{\sEnd}[1]{\mathcal{E}nd(#1)}

\newcommand{\dd}{\mathrm{d}}
\newcommand{\hook}[1]{\iota_{#1}}

\newcommand{\BBW}{
Let $(\X,\pi)$ be a compact K\"ahler Poisson manifold, and suppose that $\bL \subset \X$ is a compact symplectic leaf whose fundamental group is finite.  Then there exist a compact K\"ahler Poisson manifold $\Y$, and a finite \'etale Poisson morphism $\bLt \times \Y \to \X$, where $\bLt$ is the universal cover of $\bL$.}

\begin{document}

\title{A global Weinstein splitting theorem for holomorphic Poisson manifolds}
\author{St\'ephane Druel\thanks{CNRS/Universit\'e Claude Bernard Lyon 1, \email{stephane.druel@math.cnrs.fr}} \and Jorge Vit\'orio Pereira\thanks{IMPA, \email{jvp@impa.br}} \and Brent Pym\thanks{McGill University, \email{brent.pym@mcgill.ca}} \and Fr\'ed\'eric Touzet\thanks{Universit\'e Rennes 1, \email{frederic.touzet@univ-rennes1.fr}}}
\maketitle

\begin{abstract}
We prove that if a compact K\"ahler Poisson manifold has a symplectic leaf with finite fundamental group, then after passing to a finite \'etale cover, it decomposes as the product of the universal cover of the leaf and some other Poisson manifold.  As a step in the proof, we establish a special case of Beauville's conjecture on the structure of compact K\"ahler manifolds with split tangent bundle. 
\end{abstract}

\section{Introduction}

In 1983, Weinstein~\cite{Weinstein1983} proved a fundamental fact about Poisson brackets, nowadays known as his ``splitting theorem'': if $p$ is a point in a Poisson manifold $\X$ at which the matrix of the Poisson bracket has rank $2k$, then $p$ has a neighbourhood that decomposes as a product of a symplectic manifold of dimension $2k$, and a Poisson manifold for which the Poisson bracket vanishes at $p$.   An important consequence of this splitting is that $\X$ admits a canonical (and possibly singular) foliation by symplectic leaves, which is locally induced by the aforementioned product decomposition.

It is natural to ask under which conditions this splitting is \emph{global}, so that $\X$ decomposes as a product of Poisson manifolds, having a symplectic leaf as a factor (perhaps after passing to a suitable covering space). If $\X$ is compact, an obvious necessary condition is that the symplectic leaf is also compact.  However, this condition is not sufficient; it is easy to construct examples in the $C^\infty$ context where both $\X$ and its leaf are simply connected, but $\X$ does not decompose as a product.

In contrast, we will show that for holomorphic Poisson structures on compact K\"ahler manifolds, the existence of such splittings is quite a general phenomenon:
\begin{theorem}\label{thm:BBW-decomp}
\BBW
\end{theorem}
Here and throughout, by a ``K\"ahler Poisson manifold'', we mean a pair $(\X,\pi)$ where $\X$ is a complex manifold that admits a K\"ahler metric, and $\pi \in \coH[0]{\X,\wedge^2 \cT{\X}}$ is a holomorphic bivector that is Poisson, i.e.~the Schouten bracket $[\pi,\pi]=0$.  We will not make reference to any specific choice of K\"ahler metric at any point in the paper.

We remark that all hypotheses of \autoref{thm:BBW-decomp} are used in an essential way in the proof.  Indeed, in \autoref{sec:BBW-decomp}, we give examples showing that the conclusion may fail if the Poisson manifold $\X$ is non-K\"ahler ($C^\infty$ or complex analytic), if $\X$ is not compact, or if the fundamental group of $\bL$ is infinite.

The proof of \autoref{thm:BBW-decomp} is given in \autoref{sec:BBW-decomp}.  It rests on the following result of independent interest, which we establish using Hodge theory in \autoref{sec:subcal}.

\begin{theorem}\label{thm:subcal}
Let $(\X,\piX)$ be compact K\"ahler Poisson manifold, and suppose that $i : \bL\hookrightarrow \X$ is the inclusion of a compact symplectic leaf.  Then the holomorphic symplectic form on $\bL$ extends to a global closed holomorphic two-form $\omX \in \coH[0]{\X,\forms[2]{\X}}$ of constant rank such that the composition $\theta := \piXs\omXf \in \sEnd{\cT{\X}}$ is idempotent, i.e.~$\theta^2 =\theta$.
\end{theorem}
Here  $\piXs : \forms[1]{\X} \to \cT{\X}$ and $\omXf : \cT{\X} \to \forms[1]{\X}$ are the usual maps defined by interior contraction into the corresponding bilinear forms.  Following work of Frejlich--M\u{a}rcu\cb{t} in the $C^\infty$ setting, we refer to a holomorphic two-form $\sigma$ as in \autoref{thm:subcal} as a \defn{subcalibration of $\piX$}.  The key point about a subcalibration, which they observed, is that it provides a splitting $\cT{\X} = \cF \oplus \cG$ into a pair of subbundles $\cF = \img \theta$ and $\cG = \ker \theta$ that are orthogonal with respect to $\piX$.  Moreover $\cF$ is automatically involutive, and $\cG$ is involutive if and only if the component of $\omX$ lying in $\wedge^2\cF^\vee$ is closed.  But in the compact K\"ahler setting, the latter condition is automatic by Hodge theory, and furthermore, any splitting of the tangent bundle into involutive subbundles is expected to arise from a splitting of some covering space of $\X$, according to the following open conjecture of Beauville (2000):

\begin{conjecture}[{\cite[(2.3)]{Beauville2000}}]\label{conj:beauville}
Let $\X$ be a compact K\"ahler manifold equipped with a holomorphic decomposition $\cT{\X} = \bigoplus_{i\in I} \cF_i$ of the tangent bundle such that each subbundle $\bigoplus_{j\in J} \cF_j$, for $J \subset I$, is involutive.  Then the universal cover of $\X$ is isomorphic to a product $\prod_{i \in I} \U_i$ in such a way that the given decomposition of $\cT{\X}$ corresponds to the natural decomposition $\cT{\prod_{i \in I}\U_i} \cong \bigoplus_{i \in I} \cT{\U_i}$.
\end{conjecture}

Thus, our second key step in the proof of \autoref{thm:BBW-decomp} is to prove the following theorem, which establishes a special case of \autoref{conj:beauville}, but with a stronger conclusion:
\begin{theorem}\label{thm:split}
Suppose that $\X$ is a compact K\"ahler manifold equipped with a splitting $\cT{\X} = \cF \oplus \cG$ of the tangent bundle into involutive subbundles.  If $\cF$ has a compact leaf $\bL$ with finite holonomy group, then the splitting of $\cT{\X}$ is induced by a splitting of the universal cover of $\X$ as a product of manifolds.  If, in addition, $\bL$ has finite fundamental group and trivial canonical class, then the splitting of the universal cover is induced by a splitting of a finite \'etale cover.
\end{theorem}
The proof of \autoref{thm:split} is given in \autoref{sec:beauville}; see in particular \autoref{thm:Ehresmann} and \autoref{cor:split-CY-leaf}.  It exploits the second-named author's global Reeb stability theorem for holomorphic foliations on compact K\"ahler manifolds~\cite{Pereira01}, which implies that all leaves of $\cF$ are compact with finite holonomy.  We then apply the theory of holonomy groupoids to show that the holonomy covers of the leaves assemble into a bundle of K\"ahler manifolds over $\X$, equipped with a flat Ehresmann connection induced by the transverse foliation $\cG$.  Finally, we exploit Lieberman's structure theory for automorphism groups of compact K\"ahler manifolds~\cite{Lieberman78} to analyze the monodromy action of the fundamental group of $\X$ on the fibres, and deduce the result.

The results above have  several interesting consequences for the global structure of compact K\"ahler Poisson manifolds.  For instance \autoref{thm:BBW-decomp} immediately implies the following statement.

\begin{corollary}\label{cor:iso}
If $(\X,\pi)$ is a compact K\"ahler Poisson manifold, then all simply connected compact symplectic leaves in $(\X,\pi)$ are isomorphic.
\end{corollary}

Meanwhile, we have the following immediate corollaries of \autoref{thm:subcal}:
\begin{corollary}\label{cor:irred}
Let $(\X,\piX)$ be a compact connected K\"ahler Poisson manifold such that the tangent bundle $\cT{\X}$ is irreducible. If $\bL \subset \X$ is a compact symplectic leaf, then either $\bL =\X$ or $\bL$ is a single point.
\end{corollary}

\begin{corollary}
Let $(\X,\piX)$ be a compact K\"ahler Poisson manifold such that the Hodge number $h^{2,0}(\X)$ is equal to zero.  If $\bL \subset \X$ is a compact symplectic leaf, then $\bL$ is a single point.
\end{corollary}

Note that the vanishing $h^{2,0}(\X)=0$ holds for a wide class of manifolds of interest in Poisson geometry, including all Fano manifolds, all rational manifolds, and more generally all rationally connected manifolds.  Many natural examples arising in gauge theory and algebraic geometry fall into this class.  Applied to the case in which $\X$ is a projective space, this answers a question posed by the third author about the existence of projective embeddings that are compatible with Poisson structures:
\begin{corollary}
A compact holomorphic symplectic manifold of positive dimension can never be embedded as a Poisson submanifold in a projective space, for any choice of holomorphic Poisson structure on the latter.
\end{corollary}

\paragraph{Acknowledgements:} We thank Pedro Frejlich and Ioan M\u{a}rcu\cb{t} for correspondence, and in particular for sharing their (currently unpublished) work on subcalibrations with us.  We also thank Henrique Bursztyn, Marco Gualtieri and Ruxandra Moraru for interesting discussions.  In particular \autoref{cor:iso} and \autoref{cor:irred} were pointed out by Bursztyn and Gualtieri, respectively.

This project grew out of discussions that took place during the school on ``Geometry and Dynamics of Foliations'', which was hosted May 18--22, 2020 by the Centre International de Rencontres Math\'ematiques (CIRM), as part of the second author's Jean-Morlet Chair.  We are grateful to the CIRM for their support, and for their remarkable agility in converting the entire event to a successful virtual format on short notice, in light of the COVID-19 pandemic.

S.D.~was supported by the ERC project ALKAGE (ERC grant Nr 670846).  S.D., J.V.P.~and F.T.~were supported by CAPES-COFECUB project Ma932/19.  S.D.~and~F.T.~were supported by the ANR project Foliage (ANR grant Nr ANR-16-CE40-0008-01). J.V.P.~was supported by CNPq, FAPERJ, and CIRM.  B.P.~was supported by a faculty startup grant at McGill University, and by the Natural Sciences and Engineering Research Council of Canada (NSERC), through Discovery Grant RGPIN-2020-05191.

\section{Subcalibrations of K\"ahler Poisson manifolds}
\label{sec:subcal}

Throughout this section, we fix a connected complex manifold $\X$ and a holomorphic Poisson structure on $\X$, i.e.~a holomorphic bivector $\piX \in \coH[0]{\X,\Tpol[2]{\X}}$ such that the Schouten bracket $[\piX,\piX]$ vanishes identically.  We recall that the \defn{anchor map} of $\piX$ is the $\cO{\X}$-linear map
\[
\piXs : \forms[1]{\X} \to \cT{\X}
\]
given by contraction of forms into $\piX$.  Its image is an involutive subsheaf, defining a possibly singular foliation of $\X$.  If $i : \bL \hookrightarrow \X$ is a leaf of this foliation, then $\piL := \piX|_\bL$ is a nondegenerate Poisson structure on $\bL$, so that its inverse
\[
\omL := \piL^{-1} \in \coH[0]{\bL,\forms[2]{\bL}}
\]
is a holomorphic symplectic form.  The pair $(\bL,\omL)$ is called a \defn{symplectic leaf} of $(\X,\pi)$.

\subsection{Extending forms on symplectic leaves}

 The following lemma gives a sufficient condition for the holomorphic symplectic form on a symplectic leaf to extend to all of $\X$.

\begin{lemma}\label{lem:extend}
Let $(\X,\pi)$ be a compact K\"ahler Poisson manifold, and suppose that $(\bL,\eta)$ is a compact symplectic leaf with inclusion $i : \bL \hookrightarrow \X$.  Then there exists a global closed holomorphic two-form $\omXo \in \coH[0]{\X,\forms[2]{\X}}$ such that $i^*\omXo = \omL$.
\end{lemma}

\begin{proof}
This  proof  is a variant of the arguments in \cite[Proposition 3.1]{Druel2019} and \cite[Theorem 5.6]{Loray2018}.  Suppose that $\dim \bL = 2k$ and let $\omega \in \coH[1]{\X,\forms[1]{\X}} \cong \coH[1,1]{\X}$ be any K\"ahler class.  Note that since $\pi$ is holomorphic, the contraction operator $\hook{\pi}$ on $\forms{\X}$ descends to the Dolbeault cohomology $\coH{\X,\forms{\X}}$.  In particular, we have a well-defined class
\[
\alpha := \hook{\pi}^k \omega^{2k} \in \coH[2k]{\X,\cO{\X}}
\]
We claim that $i^*\alpha \in \coH[2k]{\bL,\cO{\bL}}$ is nonzero.  Indeed, since $\pi|_\bL = \piL$, we have the following commutative diagram:
\[
\xymatrix{
\coH[2k]{\X,\forms[2k]{\X}} \ar[r]^-{\hook{\piX}^k} \ar[d]^-{i^*} & \coH[2k]{\X,\cO{\X}} \ar[d]^-{i^*} \\
\coH[2k]{\bL,\forms[2k]{\bL}} \ar[r]^-{\hook{\piL}^k} & \coH[2k]{\bL,\cO{\bL}}
}
\]
The bottom arrow is an isomorphism since $\piL^k$ is a trivialization of the anticanonical bundle of $\bL$.  Meanwhile $i^*\omega^{2k} \in \coH[2k]{\bL,\forms[2k]{\bL}}$ is nonzero since $i^*\omega$ is a K\"ahler class on $\bL$.  It follows that $i^*\alpha = \hook{\piL}^ki^*\omega^{2k} \ne 0$ as claimed.

Using the Hodge symmetry $\overline{\coH[2k]{\X,\cO{\X}}} \cong \coH[0]{\X,\forms[2k]{\X}}$, the complex conjugate of $\alpha$ gives a holomorphic $2k$-form
\[
\mu \in \coH[0]{\X,\forms[2k]{\X}}
\]
such that $i^*\mu \in \coH[0]{\bL,\forms[2k]{\bL}}$ is nonzero. Since the canonical bundle of $\bL$ is trivial, $i^*\mu$ must be a constant multiple of the holomorphic Liouville volume form associated with the holomorphic symplectic structure on $\bL$.  Hence by rescaling $\mu$, we may assume without loss of generality $i^*\mu = \frac{1}{k!}\omL^k$.  Now consider the global holomorphic two-form
\[
\omXo := \tfrac{1}{(k-1)!}\hook{\piX}^{k-1}\mu,
\]
which is closed by Hodge theory, since $\X$ is compact K\"ahler.  We claim that it restricts to the symplectic form on $\bL$.  Indeed,
\[
i^*\omXo = \tfrac{1}{(k-1)!}i^*(\hook{\piX}^{k-1}\mu) = \tfrac{1}{(k-1)!}\hook{\piL}^{k-1}i^*\mu = \tfrac{1}{(k-1)!k!}\hook{\piL}^{k-1}\omL^k = \omL
\]
as desired.
\end{proof}

The lemma only gives the existence of the holomorphic two-form $\omXo$, but says little about its global properties.  Nevertheless, it can be used to produce a two-form with the following property, the ramifications of which are explained in \autoref{sec:subcal-split} below:
\begin{definition}[Frejlich--M\u{a}rcu\cb{t}]
A holomorphic two-form $\omX \in \coH[0]{\X,\forms[2]{\X}}$ is called a \defn{subcalibration of $(\X,\pi)$} if it is closed, and the composition
\[
\theta := \piXs\omXf \in \sEnd{\cT{\X}}
\]
is idempotent (i.e.~$\theta^2=\theta$), where $\omXf : \cT{\X} \to \forms[1]{\X}$ is the map defined by contraction of vector fields into $\omX$.
\end{definition}

We will be interested in subcalibrations that are compatible with our chosen symplectic leaf $\bL$ in the following sense:
\begin{definition}
A subcalibration $\sigma$ of $(\X,\piX)$ is \defn{compatible with the symplectic leaf $i : \bL\hookrightarrow  \X$} if $i^*\img \theta = \cT{\bL} \subset i^*\cT{\X}$.
\end{definition}
Equivalently, the subcalibration is compatible with $\bL$ if the projection of $\pi$ to $\wedge^2\img\theta$ is a constant rank bivector that is an extension of the nondegenerate Poisson structure $\pi_\bL$ on $\bL$.

The following gives a simple condition that allows an arbitrary extension of the two-form on $\bL$ to be refined to a subcalibration compatible with $\bL$.  Note that it applies, in particular, whenever $\X$ is a compact connected K\"ahler manifold:

\begin{lemma}\label{lem:subcal-correct}
Let $(\X,\pi)$ be a holomorphic Poisson manifold and let $(\bL,\omL)$ be a symplectic leaf of $(\X,\pi)$.  Suppose that the following conditions hold:
\begin{enumerate}
\item $h^0(\X,\cO{\X}) = 1$, i.e.~every global holomorphic function on $\X$ is constant
\item $\coH[0]{\X,\forms[2]{\X}} = \coH[0]{\X,\forms[2,\textrm{cl}]{\X}}$, i.e.~every global holomorphic two-form on $\X$ is closed
\item $\omL$ extends to a global holomorphic two-form on $\X$
\end{enumerate}
Then $\omL$ extends to a subcalibration compatible with $\bL$.
\end{lemma}
\begin{proof}
Let $\omXo \in \coH[0]{\X,\forms[2]{\X}}$ be any extension of the symplectic form on $\bL$ to a global holomorphic two-form on $\X$, and let $\theta_0 := \piXs\omXof \in \sEnd{\cT{\X}}$.  Since $\piL = \piX|_\bL$ is inverse to the symplectic form $\omL = i^*\omXo$ on $\bL$, it follows easily that $\theta_0|_\bL \in \sEnd{i^*\cT{\X}}$ is idempotent with image $\cT{\bL} \subset i^*\cT{\X}$, giving a splitting
\[
i^*\cT{\X} \cong \cT{\bL}\oplus \ker \theta_0|_\bL
\]
In particular, the characteristic polynomial of $\theta_0|_{\bL}$ is given by
\[
P(t) = t^{n-2k}(t-1)^{2k}
\]
where $n = \dim \X$ and $k = \tfrac{1}{2} \dim \bL$.  But the coefficients of the characteristic polynomial of $\theta_0$ are holomorphic functions on $\X$, and since $h^{0}(\X,\cO{\X}) = 1$, such functions are constant. We conclude that $P(t)$ is the characteristic polynomial of $\theta_0$ over all of $\X$. 

Taking generalized eigenspaces of $\theta_0$, we obtain a decomposition
\[
\cT{\X} \cong \cF \oplus \cG
\]
where $i^*\cF = \cT{\bL}$.  Projecting the two-form $\omXo$ to $\wedge^2\cF^\vee$ we therefore obtain a new holomorphic two-form $\omX_1$ such that $i^*\omX_1 = \omL$, which has the additional property that $\cG \subset \ker \omX_1$.   Then $\theta_1 = \piXs\omX_1^\flat$ also has characteristic polynomial $P(t)$, giving a splitting $\cT{\X} = \cF' \oplus \cG$ with respect to which $\theta_1$ takes on the Jordan--Chevalley block form
\[
\theta_1 = \begin{pmatrix}
1 + \phi  & 0\\
0 & 0
\end{pmatrix}
\]
where $\phi \in \End{\cF'}$ is nilpotent.  Using the identity   $\theta_1\piXs=\piXs\theta_1^\vee$ as maps $\forms[1]{\X} \to  \cT{\X}$,  one calculates that $\piX$ must be block diagonal, i.e.~equal to the sum of its projections to $\wedge^2\cG$ and $\wedge^2\cF'$.  The latter projection, say $\piX' \in \wedge^2\cF'$, is then nondegenerate because $\theta_1=\piXs\omX_1^\flat$ is invertible on $\cF'$.  We may therefore define a two-form $\omX := (\pi')^{-1} \in \wedge^2 \cF' \subset \forms[2]{\X}$. By construction, the endomorphism $\theta := \piXs\omX^\flat$ preserves the decomposition $\cT{\X} = \cF' \oplus \cG$, acts as the identity on $\cF'$, and has $\cG$ as its kernel.  In particular, $\theta$ is idempotent and $\omX$ restricts to the symplectic form on $\bL$.  Moreover, $\omX$ is closed by hypothesis.  Thus $\omX$ is a subcalibration of $\pi$ compatible with $\bL$, as desired.
\end{proof}

Combining \autoref{lem:extend} and \autoref{lem:subcal-correct}, we immediately obtain the following statement, which is a rephrasing of \autoref{thm:subcal} from the introduction:
\begin{corollary}\label{cor:exists-subcal}
If $(\X,\pi)$ is a compact K\"ahler manifold equipped with a holomorphic Poisson structure and $\bL \subset \X$ is a compact symplectic leaf, then there exists a subcalibration of $(\X,\pi)$ compatible with $\bL$.
\end{corollary}

\subsection{Splitting the tangent bundle}
\label{sec:subcal-split}

Note that if $\omX$ is a subcalibration of $(\X,\pi)$, then the operator $\theta := \pis\omXf$ gives a decomposition
\[
\cT{\X} = \cF \oplus \cG
\]
of the tangent bundle into the complementary subbundles
\[
\cF := \img \theta, \qquad \cG := \ker \theta.
\]
We may then project $\pi$ onto the corresponding summands in the exterior powers to obtain global bivectors
\[
\pi_\cF \in \coH[0]{\X,\wedge^2 \cF} \qquad \pi_\cG \in \coH[0]{\X \wedge^2 \cG}.
\]
Similarly, the form $\sigma$ projects to a pair of sections
\[
\omX_\cF \in \coH[0]{\X, \wedge^2 \cF^\vee} \qquad \omX_\cG \in \coH[0]{\X, \wedge^2 \cG^\vee},
\]
which we may view as global holomorphic two-forms on $\X$ via the splitting $\cT{\X}^\vee \cong \cF^\vee \oplus \cG^\vee$.

An elementary linear algebra computation then shows that
\[
\pi = \piF + \piG, \qquad \omX = \omF + \omG
\]
and $\piF$ is inverse to $\omF$ on $\cF$, so that
\[
\cF = \img \piF \qquad \textrm{and} \qquad \cG = \ker \omF.
\]
With this notation in place, we may state the following fundamental result about subcalibrations, due to Frejlich--M\u{a}rcu\cb{t}, which will play a key role in what follows:
\begin{theorem}[Frejlich--M\u{a}rcu\cb{t}]\label{thm:frelich-marcut}
Suppose that $\sigma$ is a subcalibration of $(\X,\pi)$.   Then the bivector fields $\piF,\piG$ are Schouten commuting Poisson structures, i.e.
\begin{align}
[\piF,\piF] = [\piG,\piG] = [\piF,\piG] = 0. \label{eq:commute}
\end{align}
In particular, $\cF = \img \piF$ is involutive. Moreover, $\cG$ is involutive if and only if $\omF$ is closed.
\end{theorem}

\begin{proof}The proof of Frejlich--M\u{a}rcu\cb{t} makes elegant use of the notion of a Dirac structure.  For completeness, we present here an essentially equivalent argument based on the related notion of a gauge transformation of Poisson structures.

We recall from~\cite[Section 4]{Severa2001} that if $B$ is a closed holomorphic two-form 
 such that the operator $1+B^\flat\pis \in \sEnd{\forms[1]{\X}}$ is invertible, then the gauge transformation of $\pi$ by $B$ is a new Poisson structure $B \star \pi \in \coH[0]{\X,\wedge^2\cT{\X}}$ whose underlying foliation is the same as that of $\pi$, but with the symplectic form on each leaf modified by adding the pullback of $B$.  Equivalently, $B\star\pi$ is determined by its anchor map, which is given by the formula $(B\star \pi)^\sharp = \pis(1+B^\flat \pis)^{-1} : \forms[1]{\X} \to \cT{\X}$. The skew symmetry of $B\star \pi$ follows from the skew symmetry of $\pi$ and $B$, while the identity $[B\star \pi,B\star \pi]=0$ is deduced using the closedness of $B$ and the equation $[\pi,\pi]=0$.

We apply this construction to the family of two-forms $B(t) := t\omX$.  Note that since $\theta := \pis \omXf$ is idempotent, the operator $1 + B^\flat(t)\pis = 1 + t\theta^\vee \in \sEnd{\forms[1]{\X}}$ is invertible for all $t \ne -1$.  We therefore obtain a family of holomorphic Poisson bivectors
\[
\pi(t) := B(t) \star \pi, \qquad t \in \CC\setminus\{-1\},
\]
with anchor map
\[
\pis(t) = \pis(1+t\theta^\vee)^{-1}. 
\]

Since $\theta$ is idempotent, we have $(1+t\theta^\vee)^{-1} = 1 - \tfrac{t}{1+t}\theta^\vee$, which implies that
\[
\pi(t) = \piF + \piG -\tfrac{t}{1+t}\piF  = \tfrac{1}{1+t}\piF + \piG
\]
for all $t \neq -1$.  Since $[\pi(t),\pi(t)] = 0$ for all $t \neq -1$, the bilinearity of the Schouten bracket therefore implies the desired identities \eqref{eq:commute}.   This implies immediately that $\cF = \img\piF$ is the tangent sheaf of the symplectic foliation of the Poisson bivector $\piF$; in particular, it is involutive.  

It remains to show that $\cG$ is involutive if and only if $\omF$ is closed.  To this end, observe that if $\omF$ is closed, then its kernel $\cG = \ker \omF$ is involutive by elementary Cartan calculus.
%: indeed, if $\eta_1,\eta_2 \in \cG$, then
%\[
%\hook{[\eta_1,\eta_2]}\omF = \lie{\eta_1}\hook{\eta_2}\omF - \hook{\eta_2}\lie{\eta_1}\omF = \lie{\eta_1} \eta_2 - \hook{\eta_2}\dd \hook{\eta_1}\omF - \hook{\eta_2}\hook{\eta_1}\dd \omF
%\]
%and every term on the right hand side vanishes identically, so that $[\eta_1,\eta_2] \in \ker \omF = \cG$ as desired. 
Conversely, if $\cG$ is involutive, then both $\cF^\vee$ and $\cG^\vee$ generate differential ideals in $\forms{\X}$.  This implies that the exterior derivatives of $\omF \in \wedge^2\cF^\vee$ and $\omG \in \wedge^2\cG^\vee$ lie in complementary subbundles of $\forms[2]{\X}$, namely
\[
\dd \omF \in \wedge^3 \cF^\vee \oplus (\wedge^2 \cF^\vee \otimes \cG^\vee)\qquad \dd \omG \in (\wedge^2 \cG^\vee \otimes \cF^\vee) \oplus \wedge^3\cG^\vee.
\]
Since $\dd \omF + \dd\omG = \dd \omX = 0$, we conclude that $\dd\omF = 0$, as desired.
\end{proof}

\section{Product decompositions}
\label{sec:beauville}

The subcalibrations discussed in the previous section give, in particular, a decomposition of the tangent bundle into involutive subbundles.  In this section we fix a complex manifold $\X$ and give criteria for such a decomposition of $\cT{\X}$ to arise from a decomposition of some covering of $\X$ as a product as in \autoref{conj:beauville}. The main objects are therefore the following:

\begin{definition}\label{def:complement}
Suppose that $\cF$ is a regular foliation of $\X$.  A \defn{foliation complementary to $\cF$}  is an involutive holomorphic subbundle $\cG \subset \cT{\X}$ such that $\cT{\X} = \cF \oplus \cG$.
\end{definition}
Note that the definition is symmetric: if $\cG$ is complementary to $\cF$ then $\cF$ is complementary to $\cG$.  However, in what follows, the foliations $\cF$ and $\cG$ will play markedly different roles.

\subsection{Fibrations, flat connections and suspensions}
\begin{definition}\label{def:fibration}
Let $\cF \subset \cT{\X}$ be a regular holomorphic foliation.  We say that $\cF$ is a \defn{fibration} if there exists a surjective holomorphic submersion $f : \X \to \Y$ whose fibres are the leaves of $\cF$. In this case, we call $\Y$ the \defn{leaf space} of $\cF$ and the map $f$ the \defn{quotient map}.
\end{definition}

\begin{remark}
The submersion $f: \X \to \Y$, if it exists, is unique up to  isomorphism, so there is no ambiguity in referring to $\Y$ as ``the'' leaf space of the foliation $\cF$.
\end{remark}

Suppose that $\cF$ is a fibration, and $\cG$ is a foliation complementary to $\cF$.  Then $\cG$ is precisely the data of a flat connection on the submersion $f : \X \to \Y$, in the sense of Ehresmann~\cite{Ehresmann1951}.  Recall that such a connection is \defn{complete} if it has the path lifting propery,~i.e.~given any point $y \in \Y$, any path $\gamma : [0,1] \to \Y$ starting at $\gamma(0) = y$, and any point $x$ lying in the fibre $f^{-1}(y) \subset \X$, there exists a unique path $\tgamma : [0,1] \to \X$ that is tangent to the leaves of $\cG$ and starts at the point $\tgamma(0)=x$.   If $f$ is proper, then every flat Ehresmann connection is complete in this sense.

\begin{definition}\label{def:suspension}
Suppose that $(\cF,\cG)$ is a pair consisting of a regular holomorphic foliation $\cF$ and a complementary foliation $\cG$.  We say that the pair $(\cF,\cG)$ is a \defn{suspension} if $\cF$ is a fibration for which $\cG$ defines a complete flat Ehresmann connection.
\end{definition}

Suppose that $(\cF,\cG)$ is a suspension with underlying fibration $f : \X \to \Y$, and $y \in \Y$ is a point in the leaf space of $\cF$.  Let $\bL = f^{-1}(y)$ be the fibre.  The \defn{monodromy representation} at $y$ is the homomorphism $\pi_1(\Y,y) \to \Aut{\bL}$ defined by declaring that a homotopy class $[\gamma] \in \pi_1(\Y,y)$ acts on $\bL$ by sending $x \in \bL$ to $\tgamma(1)$ where $\tgamma$ is the path lifting $\gamma$ with initial condition $\tgamma(0) = x$.  Lifting arbitrary paths in $\Y$ based at $y$ then gives a canonical isomorphism
\[
\X \cong \frac{\bL \times \tY}{\pi_1(\Y,y)}
\]
where $\tY$ is the universal cover of $\Y$ based at $y$, and $\pi_1(\Y,y)$ acts diagonally on the product.  Moreover the pullbacks of $\cF$ and $\cG$ to $\bL \times \tY$ coincide with the tangent bundle of the factors $\bL$ and $\tY$, respectively, as in \autoref{conj:beauville}.  In this way, we obtain an equivalence between suspensions $(\cF,\cG)$ and homomorphisms $\rho : \pi_1(\Y,y)\to\Aut{\bL}$, where $\Y$ and $\bL$ are complex manifolds and $y\in\Y$.

\begin{remark}\label{rmk:polarized}
We will make repeated use of the following observation: if $\X$ is compact K\"ahler, the restriction of any K\"ahler class on $\X$ to a leaf $\bL$ of $\cF$ gives a K\"ahler class $\omega \in \coH[1,1]{\bL}$ that is invariant under the monodromy representation, i.e.~the monodromy representation is given by a homomorphism
\[
\rho : \pi_1(\Y,y) \to \Autom{\omega}{\bL}
\]
where $\Autom{\omega}{\bL}$ is the group of biholomorphisms of $\bL$ that fix the class $\omega$.  The structure of $\Autom{\omega}{\bL}$ is well understood thanks to work of Lieberman~\cite{Lieberman78}, and this will allow us to control the behaviour of various suspensions.
\end{remark}

\subsection{Holonomy groupoids}

In \autoref{thm:Ehresmann} below, we will give criteria for a pair of foliations on a compact K\"ahler manifold to be a suspension.  Our key technical tool is the holonomy groupoid $\HolF$ of a regular foliation $\cF$.  We briefly recall the construction and refer the reader to  \cite{Moerdijk2003,Winkelnkemper1983} for details.  (Note that in \cite{Winkelnkemper1983}, the holonomy groupoid is called the ``graph'' of $\cF$.)

 If $x, y$ are two points on the same leaf $\bL$ of the foliation $\cF$,  and $\gamma$ is a path from $x$ to $y$ in $\bL$, then by lifting $\gamma$ to nearby leaves one obtains a germ of a biholomorphism from the leaf space of $\cF|_{\U_x}$ to the leaf space of $\cF|_{\U_y}$ where $\U_x,\U_y \subset \X$ are sufficiently small neighbourhoods of $x$ and $y$, respectively. This germ is called the holonomy transformation induced by $\gamma$. We say that two paths tangent to $\cF$ have the same holonomy class if their endpoints are the same, and they induce the same holonomy transformation.  
 
The \defn{holonomy groupoid}  $\HolF$ is the set of holonomy classes of paths tangent to $\cF$.
 It carries a natural complex manifold structure of dimension equal to $\dim \X + \rank \cF$, and comes equipped with a pair of surjective submersions $s,t : \HolF \to \X$ that pick out the endpoints of paths.  The usual composition of paths then makes $\HolF$ into a complex Lie groupoid over $\X$.

\begin{remark} In the differentiable setting, the holonomy groupoid may fail to be Hausdorff, but in the analytic setting we consider here, the Hausdorffness is guaranteed by \cite[Corollary of Proposition 2.1]{Winkelnkemper1983}. 
\end{remark}

\begin{remark}\label{rmk:kahler}
The map $(s,t) : \HolF \to \X\times \X$ is an immersion~\cite[0.3]{Winkelnkemper1983}.  Hence a K\"ahler structure on $\X$ induces a K\"ahler structure on $\HolF$ by pullback.
\end{remark}

Suppose that $x \in \X$, and let $\bL \subset \X$ be the leaf through $\X$.  We denote by $\HolF_x \subset \HolF$ the group of holonomy classes of loops based at $x$.  Since homotopic loops induce the same holonomy transformation, $\HolF_x$ is a quotient of the fundamental group $\pi_1(\bL,x)$.  Moreover, it acts freely on the fibre $s^{-1}(x)$, and the map $t$ descends to an isomorphism $s^{-1}(x)/\HolF_x \cong \bL$.   Put differently, the fibration $s^{-1}(\bL) \to \bL$ is a fibre bundle equipped with a complete flat Ehresmann connection whose horizontal leaves are the fibres $t^{-1}(y) \subset s^{-1}(\bL)$ where $y \in \bL$.  The holonomy group $\HolF_x$ is the image of the homomorphism
\begin{align}
\hol_x : \pi_1(\bL,x) \to \Aut{s^{-1}(x)} \label{eq:hol}
\end{align}
obtained by taking the monodromy of this flat connection.

Now suppose that $\cG$ is a foliation complementary to $\cF$.  Then the preimage $t^{-1}\cG \subset \cT{\HolF}$
defines a regular foliation on $\HolF$.  The leaves of this foliation are the submanifolds of the form $t^{-1}(\W)$ where $\W \subset \X$ is a leaf of $\cG$.

\begin{lemma}\label{lem:tpull-complement}
Suppose that $\cF$ is a regular foliation of $\X$ such that the map $s : \HolF \to \X$ is proper, and that $\cG$ is a foliation complementary to $\cF$.  Then the following statements hold:
\begin{enumerate}
\item The foliation $t^{-1}\cG \subset \cT{\HolF}$ defines a complete flat Ehresmann connection on  the fibration $s: \HolF \to \X$.
\item If $\bL \subset \X$ is a leaf of $\cF$, then the $t^{-1}\cG$-horizontal lifts of $\bL$ are exactly the fibres $t^{-1}(y)$ for $y \in \bL$.
\item If $x\in \X$ and $\bL$ is the leaf of $\cF$ through $x$, then the monodromy representation $\rho_x : \pi_1(\X,x) \to \Aut{s^{-1}(x)}$ of $t^{-1}\cG$ extends the holonomy representation of $\cF$ at $x$, i.e.~the following diagram commutes:
\[
\xymatrix{
\pi_1(\bL,x) \ar@{->>}[d]_-{\hol_x}\ar[r] & \pi_1(\X,x) \ar[d]^-{\rho_x} \\
\HolF_x \ar@{^(->}[r] & \Aut{s^{-1}(x)}
}
\]
\end{enumerate}
\end{lemma}

\begin{proof}
For the first statement, note that $\rank \cG = \dim \X - \rank \cF$ and the fibres of $t$ have dimension equal to $\rank \cF$.  Therefore $t^{-1}\cG$ has rank equal to $\dim \X$.  Note that $\cG$ is identified with the normal bundle of the foliation $\cF$ and therefore $t^{-1}\cG$ surjects by $s$ onto the normal bundle of every leaf.  Meanwhile every $t$-fibre surjects by $s$ onto the corresponding leaf.  Considering the ranks, it follows that $t^{-1}(\cG)$ is complementary to the fibres of $s$, defining a flat Ehresmann connection.  Since $s$ is proper, this connection is complete, as desired.

For the second statement, note that the horizontal lifts of a leaf $\bL \subset \X$ of $\cF$ are, by definition, given by intersecting the preimage $s^{-1}(\bL)$ with the leaves of $t^{-1}\cG$.  But $t$ projects $s^{-1}(\bL)$ onto $\bL$, which is complementary to $\cG$.  It follows that the intersection of the leaves of $t^{-1}\cG$ with $s^{-1}(\bL)$ are the fibres $t^{-1}(y)$ for $y \in \bL$, as claimed.  The third statement then follows immediately from the description of the holonomy representation \eqref{eq:hol} above.
\end{proof}

\subsection{Suspensions on K\"ahler manifolds}

We are now in a position to prove our main result on complementary foliations.
\begin{theorem}\label{thm:Ehresmann}
Let $\cF$ be a regular foliation on a compact K\"ahler manifold, and suppose that $\cF$ has a compact leaf $\bL \subset \X$ with finite holonomy group.  Then the following statements hold:
\begin{enumerate}\item\label{stmt:suspend} For every foliation $\cG$ complementary to $\cF$, there exists a finite \'etale cover $\phi : \tX \to \X$ such that $(\phi^{-1}\cF,\phi^{-1}\cG)$ is a suspension.
\item\label{stmt:split} If, in addition, the fundamental group of $\bL$ is finite and the universal cover $\bLt$ admits no nonzero holomorphic vector fields, i.e.~$h^0(\bLt,\cT{\bLt})=0$, then we can arrange so that the suspension in statement \ref{stmt:suspend} is trivial, i.e.~there exists a compact K\"ahler manifold $\Y$ and a finite \'etale cover
\[
\phi : \bLt \times \Y \to \X
\]
such that $\phi^{-1}\cF$ and $\phi^{-1}\cG$ are identified with the tangent bundles of the factors $\bLt$ and $\Y$, respectively.
\end{enumerate}

 \end{theorem}

As an immediate corollary of part \ref{stmt:suspend} of \autoref{thm:Ehresmann}, we obtain the following special case of Beauville's conjecture:
\begin{corollary}
\autoref{conj:beauville} holds for decompositions of the tangent bundle of the form $\cT{\X} = \cF\oplus \cG$ where $\cF$ has a compact leaf with finite holonomy.
\end{corollary}

\begin{proof}[Proof of \autoref{thm:Ehresmann}, part \ref{stmt:suspend}]
Suppose that $\X$ is a compact K\"ahler manifold and $\cF$ is a regular holomorphic foliation having a compact leaf with finite holonomy group.    Then by the global Reeb stability theorem for compact K\"ahler manifolds~\cite[Theorem 1]{Pereira01}, \emph{every} leaf of $\cF$ is compact with finite holonomy group.  As observed in \cite[Example 5.28(2)]{Moerdijk2003}, this implies that the submersions  $s,t: \HolF \to \X$ are proper maps.

If $\cG$ is any foliation complementary to $\cF$, we obtain from \autoref{lem:tpull-complement} a flat Ehresmann connection on the fibration  $s :\HolF \to \X$ whose monodromy representation induces the holonomy representation of every leaf.  Let us choose a base point $x \in \X$, and consider the monodromy representation
\[
\rho: \pi_1(\X,x) \to \Aut{s^{-1}(x)}.
\]
Note that by \autoref{rmk:kahler}, $\HolF$ is a K\"ahler manifold.  If we choose a K\"ahler class on $\HolF$, its restriction to $s^{-1}(x)$ gives a K\"ahler class $\omega \in \coH[1,1]{s^{-1}(x)}$ that is invariant under the monodromy action.  Hence $\rho$ factors through the subgroup $\Autom{\omega}{s^{-1}(x)}$ of biholomorphisms of $s^{-1}(x)$ that fix the class $\omega$.

By \autoref{lem:finite-index} below applied to the monodromy action of $\Pi := \pi_1(\X,x)$ on $\Z := s^{-1}(x)$, there exists a finite-index subgroup $\Gamma < \pi_1(\X,x)$ whose image $\rho(\Gamma) < \Autom{\omega}{s^{-1}(x)}$ contains no finite subgroups that act freely on $s^{-1}(x)$.  Let $\phi : \tX \to \X$ be the covering space corresponding to $\Gamma$, with base point $\tx \in \tX$ chosen so that $\phi_*\pi_1(\tX,\tx) = \Gamma$.  Then $\phi$ is a finite \'etale cover since $\Gamma$ has finite index.

We claim that the pair $(\phi^{-1}\cF,\phi^{-1}\cG)$ is a suspension.  Note that to prove this, it suffices to show that the holonomy groups of the leaves of $\phi^{-1}\cF$ are trivial, for in this case the holonomy groupoid of $\phi^*\cF$ embeds as the graph of an equivalence relation in $\tX \times \tX$, which in turn implies that $\phi^{-1}\cF$ is a fibration whose quotient map is proper.  Then the complementary foliation $\phi^{-1}\cG$ is a flat Ehresmann connection, as desired, and this establishes part \ref{stmt:suspend} of the theorem.

To see that the holonomy groups of the leaves of $\phi^{-1}\cF$ are indeed trivial, suppose that $\bLt\subset \tX$ is a leaf of $\phi^{-1}\cF$.  Then $\phi$ restricts to an \'etale cover $\bLt \to \bL$ where $\bL$ is a leaf of $\cF$.  If we choose a base point $\ty \in \bLt$ with image $y = \phi(\ty) \in \bL$, then the germ of the leaf space of $\phi^{-1}\cF$ at $\ty$ is identified with the germ of the leaf space of $\cF$ at $y$, so that the holonomy group of $\bLt$ at $\ty$ is canonically identified with  the image of the composition
\[
\xymatrix{
\pi_1(\bLt,\ty) \ar[r] & \pi_1(\bL,y) \ar[r] & \HolF_y. %\ar@{->}[r]&  \Aut{s^{-1}(y)}.
}
\]
Choose a homotopy class of a path $\tgamma$ from $\tx$ to $\ty$, let $\gamma$ be its projection to a homotopy class from $x$ to $y$, and let $\hol_\gamma : s^{-1}(y) \to s^{-1}(x)$ be the isomorphism given by parallel transport of the Ehresmann connection on $\HolF \to \X$.  The adjoint actions of $\tgamma,\gamma$ and $\hol_\gamma$ then fit in a commutative diagram
\[
\xymatrix{
\pi_1(\bLt,\ty) \ar[r]\ar[d]^-{\Ad_{\tgamma}} & \pi_1(\bL,y) \ar[r]\ar[d]^-{\Ad_{\gamma}} & \HolF_y\ar@{-->}[d] \ar@{^(->}[r] & \Aut{s^{-1}(y)} \ar[d]^{\Ad_{\hol_{\gamma}}} \\
\pi_1(\tX,\tx) \cong \Gamma \ar@{^(->}[r] & \pi_1(\X,x) \ar[r] & \Autom{\omega}{s^{-1}(x)} \ar@{^(->}[r] & \Aut{s^{-1}(x)}
}
\]
Note that since $\HolF_y$ acts freely on $s^{-1}(y)$, and the holonomy $\hol_\gamma$ is an isomorphism between the fibres, the conjugate of $\HolF_y$ by $\hol_\gamma$ acts freely on $s^{-1}(x)$.  We conclude that the dashed arrow embeds the holonomy group of $\bLt$ as a finite subgroup in $\Autom{\omega}{s^{-1}(x)}$ that acts freely on $s^{-1}(x)$.  Since the diagram is commutative, this subgroup lies in the image of $\Gamma$ and hence by our choice of $\Gamma$ it must be trivial, as claimed.
\end{proof}

\begin{proof}[Proof of \autoref{thm:Ehresmann}, part \ref{stmt:split}]
By part \ref{stmt:suspend} of the theorem, we may assume without loss of generality that $(\cF,\cG)$ is a suspension with base $\Y$ and typical fibre $\bL'$, where $\bL'$ is a covering space of $\bL$. Moreover the monodromy representation of this suspension takes values in the group $\Autom{\omega'}{\bL'}$ of biholomorphism of $\bL'$ that fix some K\"ahler class $\omega' \in \coH[1,1]{\bL'}$.

Now note that if the fundamental group of $\bL$ is finite and the universal cover of $\bL$ has no nonvanishing vector fields, then $\coH[0]{\bL',\cT{\bL'}}=0$ as well.  Therefore $\Autom{\omega'}{\bL'}$ is finite by \cite[Proposition 2.2]{Lieberman78}.  The kernel of the monodromy representation therefore gives a finite index subgroup of $\pi_1(\Y)$, and passing to the corresponding finite \'etale cover of $\Y$, we trivialize the suspension, giving an \'etale map $\bL'\times \Y \to \X$ that implements the desired splitting of the tangent bundle.  Then passing to the universal cover of $\bL'$, we obtain the desired statement.
\end{proof}

\begin{lemma}\label{lem:finite-index}
Let $\Z$ be a compact K\"ahler manifold with K\"ahler class $\omega \in \coH[1,1]{\Z}$, let $\Pi$ be a finitely generated group, and let $\rho : \Pi \to \Autom{\omega}{\Z}$ be a homomorphism to the group of biholomorphisms of $\Z$ that fix the class $\omega$.  Then there exists a finite-index subgroup $\Gamma < \Pi$ whose image $\rho(\Gamma) < \Autom{\omega}{\Z}$ contains no finite subgroups that act freely on $\Z$.
\end{lemma}

\begin{proof}
By taking preimages under $\rho$, we reduce the problem to the case where $\rho$ is injective, so we may assume without loss of generality that $\Pi < \Autom{\omega}{\Z}$ is a subgroup.  Moreover, by \cite[Proposition 2.2]{Lieberman78}, the neutral component $\Aut{\Z}_0$ of the group of all biholomorphisms of $\Z$ is a finite index subgroup in $\Autom{\omega}{\Z}$.   In particular, $\Pi \cap \Aut{\Z}_0$ has finite index in $\Pi$, and is therefore finitely generated by Schreier's lemma.  Hence we may assume without loss of generality that $\Pi < \Aut{\Z}_0$.

By \cite[Theorems 3.3, 3.12 and 3.14]{Lieberman78}, there is an exact sequence
\begin{align}
\xymatrix{
1 \ar[r] & \N \ar[r] & \Aut{\Z}_0 \ar[r] & \T \ar[r] & 1
}\label{eq:lieberman-exact}
\end{align}
where $\N$ is the closed subgroup exponentiating the Lie algebra of holomorphic vector fields with nonempty zero locus, and $\T$ is a compact complex torus (a finite connected \'etale cover of the Albanese torus of $\Z$).

Since $\Pi$ is finitely generated, its image in $\T$ is a finitely generated abelian group.  Therefore, by the classification of finitely generated abelian groups, there is a finite-index subgroup $\Gamma < \Pi$ whose image in $\T$ is torsion-free.  Suppose that $\G < \Gamma$ is a finite subgroup that acts freely on $\Z$.  We claim that $\G$ is trivial.  Indeed, by construction, the image of $\G$ in $\T$ is a torsion-free finite group, hence trivial. We must therefore have that $\G < \N$.  But by \autoref{lem:fixed-points} below, every element of $\N$ has a fixed point, and therefore the only subgroup of $\N$ that acts freely is the trivial subgroup.
\end{proof}

\begin{lemma}\label{lem:fixed-points}
Let $\N$ be a connected complex Lie group, and let $\X$ be a compact complex manifold.  Let $a : \N \times \X \to \X$ be an action of $\N$ on $\X$ by biholomorphisms.  Suppose that the vector fields generating the action all have a non-empty vanishing locus.  Then every element of $\N$ has a fixed point in $\X$.
\end{lemma}

\begin{proof}
Let $\U\subset \N$ be the set of elements with at least one fixed point. Note that $\U = p(a^{-1}(\Delta))$, where $\Delta \subset \X\times \X$ is the diagonal and $p : \N \times \X \to \N$ is the projection.  Since $a$ is holomorphic, $a^{-1}(\Delta) \subset \N \times \X$ is a closed analytic subspace.  Since $p$ is proper, the Grauert direct image theorem  implies that the image $p(a^{-1}(\Delta)) = \U \subset \N$ is a closed analytic subvariety.  But we assume that every generating vector field for the action has at least one zero; therefore $\U$ contains the image of the exponential map of $\N$, and in particular it contains an open neighbourhood of the identity.  It follows that $\dim \U = \dim\N$, and therefore $\U = \N$ since $\N$ is connected.
\end{proof}

\subsection{Foliations with trivial canonical class}

In the particular case when the foliation $\cF$ or its compact leaf $\bL$ has trivial canonical class, we can strengthen the results of the previous section.  For example, we have the following:
\begin{corollary}\label{cor:split-CY-leaf}
Suppose $\X$ is a compact K\"ahler manifold, and $\cF \subset \cT{\X}$ is a regular foliation having a compact leaf $\bL$ with finite fundamental group and trivial canonical class $c_1(\bL) = 0$.  If $\cG$ is any complementary foliation, then there exists a finite \'etale cover $\phi : \bLt \times \Y \to \X$ inducing the splitting of the tangent bundle of $\X$ as in \autoref{thm:Ehresmann}, part \ref{stmt:split}.
\end{corollary}

\begin{proof}
By \autoref{thm:Ehresmann}, part \ref{stmt:split}, it suffices to show that $h^0(\bLt,\cT{\bLt})=0$ where $\bLt$ is the universal cover of $\bL$, but this vanishing is well known.  Indeed, in this case, the canonical bundle of $\bLt$ is trivial, so that $\cT{\bLt} \cong \forms[n-1]{\bLt}$ where $n = \dim \bL$. We then have
\[
h^{0}(\bLt,\cT{\bLt}) = h^0(\bLt,\forms[n-1]{\bLt}) = h^{n-1}(\bLt,\cO{\bLt}) = h^1(\bLt,\cO{\bLt}) = \tfrac{1}{2}\dim \coH[1]{\bLt;\CC} = 0
\]
by Hodge symmetry, Serre duality, the Hodge decomposition theorem and the simple connectivity of $\bLt$.
\end{proof}

This corollary, in turn has consequences for the foliation itself:
\begin{corollary}
Suppose that $\X$ is a compact K\"ahler manifold and that $\cF$ is a (possibly singular) foliation with trivial canonical class $c_1(\cF) = 0$.  If $\cF$ has a compact leaf $\bL$ with finite fundamental group, then there exists a finite \'etale cover $\phi : \bLt \times \Y \to \X$ such that $\phi^{-1}\cF = \cT{\bLt}$.
\end{corollary}

\begin{proof}
By \cite[Theorem 5.6]{Loray2018}, $\cF$ is automatically regular and admits a complementary foliation, so \autoref{cor:split-CY-leaf} applies.
\end{proof}

In the situation where  $c_1(\cF) =0$ but the compact leaf $\bL$ has infinite fundamental group, the situation becomes more complicated.  However, the Beauville--Bogomolov decomposition theorem~\cite{Beauville1983,Bogomolov1974} implies that $\bL$ has a finite \'etale cover of the form $\Z \times \T$, where $\Z$ is a simply connected compact K\"ahler manifold with $c_1(\Z) = 0$, and $\T$ is a compact complex torus.  Moreover, amongst all such coverings there is a ``minimal'' one through which all others factor.  This minimal split covering is  unique up to a non-unique isomorphism~\cite[Proposition 3]{Beauville82}.   By exploiting this result we can split the leaves of $\cF$ in a uniform fashion:

\begin{proposition}\label{C:split1}
Let $\X$ be a compact K\"{a}hler manifold and let $\cF$ be a possibly singular foliation on $\X$ with $c_1(\cF)=0$. If $\bL\subset \X$ is a compact leaf whose holonomy group is finite, then there exists a simply connected compact K\"ahler manifold $\Z$ with $c_1(\Z)=0$,  a compact complex torus $\T$, a locally trivial fibration $f\colon \W \to \Y$ between complex K\"ahler manifolds with typical fibre $\T$, and a finite \'etale cover $$\phi\colon \Z \times \W \to \X$$
such that $\phi^{-1}\cF$ is given by the fibres of the natural morphism $\Z \times \W \to \Y$ induced by $f$. Moreover if $\cG$ is any foliation complementary to $\cF$, then we may choose $\phi$ so that $\phi^{-1}\cG$ is the pullback of a flat Ehresmann connection on $f$.
\end{proposition}

\begin{proof}
By \cite[Theorem 5.6]{Loray2018}, $\cF$ is regular and there exists a foliation $\cG$ complementary to $\cF$. Then by \autoref{thm:Ehresmann} part \ref{stmt:suspend}, we may assume without loss of generality that $(\cF,\cG)$ is a suspension with base $\Y$ and typical fibre $\bL$.

 Let $\Z \times \T \to \bL$ be the minimal split cover of $\bL$ as in \cite[Section 3]{Beauville82}.  Since every automorphism of $\Z \times \T$ respects the product decomposition~\cite[Section 3, Lemma]{Beauville82}, the foliations on $\Z \times \T$ given by the tangent bundles of the factors descend to canonical foliations on $\bL$ that split the tangent bundle $\cT{\bL}$. Then, since $f$ is a locally trivial fibration, we obtain a decomposition $\cF \cong \cF_\Z \oplus \cF_\T$ into involutive subbundles with compact leaves.  Note that by construction, the leaves of $\cF_\Z$ are finite quotients of $\Z$ and therefore $\cF_\Z$ satisfies the hypotheses of \autoref{cor:split-CY-leaf}.  Hence by passing to an \'etale cover, we may assume that $\X = \Z \times \W$, so that $\cF_{\Z}$ is identified with the tangent bundle of $\Z$, and $\cF_{\T} \oplus \cG$ is identified with the tangent bundle of $\W$.

This reduces the problem to the case in which $\cF_\Z = 0$, or equivalently $\bL$ is a finite quotient of a torus $\T$, and $\cF$ is the suspension of a representation
\[
\rho : \pi_1(\Y,y) \to \Autom{\omega_\bL}{\bL} 
\]
for some K\"ahler class $\omega_\bL \in \coH[1,1]{\bL}$ and some base point $y \in \Y$.  Let $\omega \in \coH[1,1]{\T}$ be the induced K\"ahler class on $\T$.  By \autoref{lem:torus-lift} below, there exists a finite index subgroup  $\Gamma < \pi_1(\Y,y)$ such that $\rho|_\Gamma$ lifts to a homomorphism 
\[
\widetilde \rho : \Gamma \to \Autom{\omega}{\T}
\]
Let $\tY \to \Y$ be the covering determined by $\Gamma$, and let $\tX \to \tY$ be the suspension determined by $\widetilde \rho$.  Then we have a natural \'etale cover $\tX \to \X$ lifting $\tY \to \Y$ whose restriction to each fibre corresponds to the quotient map $\T \to \bL$, giving the result.
\end{proof}

\begin{lemma}\label{lem:torus-lift}
Let $\bL$ be a compact K\"ahler manifold equipped with a finite \'etale cover $\T \to \bL$, where $\T$ is a compact complex torus.  Let $\omega \in \coH[1,1]{\bL}$ be any K\"ahler class.  If $\Pi$ is a finitely generated group and $\rho : \Pi \to \Autom{\omega}{\bL}$ is a homomorphism, then there exists a finite-index subgroup $\Gamma < \Pi$ whose action on $\bL$ lifts to an action on $\T$.
\end{lemma}

\begin{proof}
As in  the proof of \autoref{lem:finite-index} we may assume without loss of generality that $\rho$ is the inclusion $\Pi \hookrightarrow  \Aut{\bL}_0 < \Autom{\omega}{\bL}$ of a subgroup of the neutral component of the full automorphism group.
By a theorem of Lichnerowicz~\cite{Lichnerowicz1967}, $\Aut{\bL}_0$ is a torus. In particular, $\Pi$ is a finitely generated abelian group, and hence it has a free abelian subgroup $\Gamma < \Pi$ of finite index.
  
Let $\T_1 \to \bL$ be the minimal split cover of $\bL$ as in \cite[Section 3]{Beauville82}, and let $\T \to \T_1$ be a factorization of $\T \to \bL$ through $\T_1 \to \bL$. We may assume without loss of generality that  $\T \to \T_1$ is a morphism of tori. Let $\omega_{\T_1}$ be the pullback of the class $\omega$ to $\T_1$.
Since the minimal split cover is unique up to isomorphism, any automorphism of $\bL$ that fixes the class $\omega$ extends to an automorphism of $\T_1$ that fixes $\omega_{\T_1}$.

Let $\G < \Autom{\omega_{\T_1}}{\T_1}$ be the subgroup of automorphisms of $\T_1$ that lift automorphisms in $\Aut{\bL}_0$.  
Then $\G$ is a complex Lie subgroup of $\Autom{\omega_{\T_1}}{\T_1}$ and the natural map $\G \to \Aut{\bL}_0$ is a surjective morphism of complex Lie groups with finite kernel. This in turn implies that the neutral component $\G_0 < \Aut{\T_1}_0$ of $\G$ is a torus where $\Aut{\T_1}_0 \cong \T_1$ denotes the neutral component of the automorphism group of $\T_1$. In particular, $\G$ is abelian.
It follows that $\Gamma < \Aut{\bL}_0$ lifts to an inclusion $\Gamma < \G < \Aut{\T_1}_0$.

Let now $\Aut{\T}_0 \cong \T$ be the neutral component of the automorphism group of $\T$. Since the map $\T \to \T_1$ is a morphism of complex Lie groups, we have a surjective natural morphism $\Aut{\T}_0 \to \Aut{\T_1}_0$ of complex Lie groups. Hence the inclusion $\Gamma < \Aut{\T_1}_0$ lifts to an inclusion $\Gamma < \Aut{\T}_0$, as desired.
\end{proof}

\begin{remark}

If one keeps the hypotheses of \autoref{C:split1} and  further assumes that $\X$ is projective, then there exists a finite \'etale cover of $\X$  isomorphic to a product where the foliation $\cF$ becomes the relative tangent bundle of the projection to one of the factors; this follows by combining \autoref{C:split1} with the fact that there exists a fine moduli scheme for polarized abelian varieties of dimension $g$ with level $N$ structures, provided that $N$ is large enough.  This gives a simpler proof of \cite[Theorem 5.8]{Loray2018}.
\end{remark}

\section{Global Weinstein splitting}
\label{sec:BBW-decomp}

We now combine the results of the previous sections to establish our main result (\autoref{thm:BBW-decomp} from the introduction), whose statement we now recall:

\begin{theorem*}
\BBW
\end{theorem*}

\begin{proof}
By \autoref{cor:exists-subcal}, there exists a subcalibration $\omX \in \coH[2]{\X,\forms[2]{\X}}$ compatible with $\pi$.  Let $\cT{\X} = \cF \oplus \cG$ be the corresponding splitting of the tangent bundle as in \autoref{sec:subcal-split}, and let $\pi = \piF +\piG$ and $\omX = \omF+\omG$ be the corresponding decompositions.  Note that $\omF$ is closed since it is holomorphic, and $\X$ is a compact K\"ahler manifold.  Therefore by \autoref{thm:frelich-marcut}, $\cF$ and $\cG$ are involutive.  Moreover, since $\bL$ is a holomorphic symplectic manifold, we have $c_1(\bL)=0$, so by \autoref{cor:split-CY-leaf} there exists an \'etale cover $\bLt \times \Y \to \X$ that induces the given splitting of the tangent bundle.

It remains to verify that the induced Poisson structure on this covering space is the sum of the pullbacks of Poisson structures on the factors.  But since $\bLt \times \Y$ is compact, this property follows immediately from the K\"unneth decomposition
\begin{align*}
\coH[0]{\bLt\times \Y, \wedge^2 \cT{\bLt\times \Y}}  \cong \coH[0]{\bLt,\wedge^2 \cT{\bLt}} \, \oplus \, \rbrac{\coH[0]{\bLt,\cT{\bLt}} \otimes \coH[0]{\Y,\cT{\Y}}} \, \oplus \, \coH[0]{\Y,\wedge^2\cT{\Y}}, \label{eqn:kunneth}
\end{align*}
and the $\pi$-orthogonality of the factors, which ensures that the induced bivector projects trivially to the middle summand in this decomposition.
\end{proof}

We conclude the paper by giving some examples which demonstrate that the conclusion of \autoref{thm:BBW-decomp} may fail if the hypotheses are weakened.
\begin{example}
The analogue of \autoref{thm:BBW-decomp} fails in the $C^\infty$ or real analytic contexts, even for Poisson structures of constant rank.

For instance, any $C^\infty$ symplectic fibre bundle (as in \cite[Chapter 6]{McDuff1998}) defines a regular Poisson manifold for which the symplectic leaves are the fibres.  Such bundles need not be trivial, even if the base and fibres are simply connected.  The simplest nontrivial example is the nontrivial $S^2$-bundle over $S^2$ underlying the odd Hirzebruch surfaces, equipped with the $C^\infty$ Poisson structure induced by a fibrewise K\"ahler form. This four-manifold is simply connected but is not diffeomorphic to $S^2\times S^2$.

Note that given a symplectic fibre bundle, we may rescale the symplectic form on the fibres by the pullback of an arbitrary nonvanishing function on the base, to obtain a Poisson manifold whose symplectic leaves are pairwise non-symplectomorphic.  In particular, even when the underlying manifold splits as a product, the Poisson structure need not decompose as a product of Poisson structure on the factors.
\end{example}

\begin{example}
The conclusion of \autoref{thm:BBW-decomp} can fail if the K\"ahler condition is dropped.  For instance, by taking the mapping torus of an infinite-order holomorphic symplectic automorphism of a K3 surface, we may construct a holomorphic symplectic fibre bundle whose total space is non-K\"ahler.  It splits only after passing to the universal cover, which has infinitely many sheets.
\end{example}

\begin{example}
The conclusion of \autoref{thm:BBW-decomp} can fail if the leaf $\bL$ has infinite fundamental group.  For example, let $\ZZ^{2n}\cong\Lambda \subset \CC^{2n}$ be a lattice, and let $\Y$ be a compact K\"ahler Poisson manifold on which $\Lambda$ acts by holomorphic Poisson isomorphisms.  Equip $\CC^{2n}$ with the standard holomorphic Poisson structure in Darboux form.  Then the quotient $\X := (\CC^{2n}\times \Y)/\Lambda$ is a compact holomorphic Poisson manifold, which is a flat fibre bundle over the symplectic base torus $\bL := \CC^{2n}/\Lambda$.  If $p \in \Y$ is a point where the Poisson structure vanishes, then $\widetilde\bL := (\CC^{2n}\times \Lambda \cdot p)/\Lambda$ defines a symplectic leaf of $\X$ for which the projection $\widetilde \bL \to \bL$ is a local diffeomorphism of holomorphic Poisson manifolds.  Two possibilities are of note: i) if $p$ is a $\Lambda$-fixed point, then $\widetilde \bL \cong \bL$, and ii) if the orbit of $p$ is infinite, then $\bL$ is non-compact.

Note that if both i) and ii) occur for some points $p \in \Y$, then $\X$ cannot split.  It is easy to construct examples of this phenomenon, e.g.~take $\Y=\PP^2$ equipped with a Poisson bivector given by an anticanonical section vanishing on the standard toric boundary divisor (a triangle).  Let $\Lambda$ act by irrational rotations of the torus.  Then the vertices of the triangle are fixed, and the smooth points of the triangle are symplectic leaves with infinite orbits.  Moreover, in this case $\X$ is K\"ahler.
\end{example}

\begin{example}
The conclusion of \autoref{thm:BBW-decomp} can fail if $\X$ is not compact. Let $\Y$ be a smooth projective manifold and let $f : \Y \to \B$ be a non-isotrivial fibration whose general fibres are $K3$ surfaces. If $\U$ is a sufficiently small
euclidean open subset of $\B$ that does not intersect the critical locus of $f$, and $\X = f^{-1}(\U)$ then $\X$ is K\"ahler, simply-connected, and admits a holomorphic Poisson structure with symplectic leaves given by the fibres of $f$, but it does not split as a product. \end{example}

\bibliographystyle{hyperamsplain}
\bibliography{bbw}
\end{document}